\documentclass[a4paper,10pt]{article}
\usepackage[utf8x]{inputenc}
\usepackage{textcomp}
\usepackage{graphicx}
\usepackage{color}
\usepackage[colorlinks]{hyperref}
\usepackage{lscape}
\usepackage[centertags]{amsmath}
\usepackage{amssymb}
\usepackage{amsthm}
\usepackage{newlfont}
\usepackage{amsfonts}
\theoremstyle{plain}
\newtheorem{thm}{Theorem}
\newtheorem{cor}{Corollary}
\newtheorem{lem}{Lemma}
\newtheorem{prop}{Proposition}
\newtheorem{conjec}{Conjecture}
\newtheorem{defn}{Definition}

\newtheorem{prob}{Problem}

\begin{document}

\begin{center}
\Large \textbf{A contribution to the second neighborhood problem}
\end{center}
\begin{center}
Salman GHAZAL\footnote{Department of Mathematics, Faculty of Sciences I, Lebanese University, Hadath, Beirut, Lebanon.\\
                       E-mail: salmanghazal@hotmail.com\\
                       Institute Camille Jordan, Département de Mathématiques, Universit\'{e} Claude Bernard Lyon 1, France.\\
                       E-mail: salman.ghazal@math.univ-lyon1.fr
                       
                       }
\end{center}
\vskip1cm
\begin{abstract}
 Seymour's Second Neighborhood Conjecture asserts that every digraph (without digons) has a vertex whose first out-neighborhood
is at most as large as its second out-neighborhood. It is proved for tournaments, tournaments missing a matching and tournaments
missing a generalized star. We prove this conjecture for classes of digraphs whose missing graph is a comb, a complete
graph minus 2 independent edges, or a complete graph minus the edges of a cycle of length 5.
\end{abstract}

\begin{section}{Introduction}

\par \hskip0.6cm In this paper, graphs are finite and simple. Directed graphs (digraphs) are orientations of graphs, so they do not contain loops,
parallel arcs, or digons (directed cycles of length 2). Let $G=(V,E)$ be a graph. The neighborhood of a vertex $v$ in $G$ is denoted by $N_G(v)$ 
and its degree is $d_G(v)=|N_G(v)|$. For $A\subseteq V$, $N_G(A)$ denotes the set of negihbors outside $A$ of the elements of $A$.
Let $D=(V,E)$ denote a digraph with vertex set $V$ and arc set $E$. As usual, $N^{+}_D(v)$ (resp. $N^{-}_D(v)$) 
denotes the (first) out-neighborhood (resp. in-neighborhood) of a vertex $v\in V$. $N^{++}_D(v)$ (resp. $N^{--}_D(v)$) denotes the second 
out-neighborhood (in-neighborhood) of $v$, which is the set of vertices that are at distance 2 from $v$ (resp. to $v$). We also denote 
$d^{+}_D(v)=|N^{+}_D(v)|$, $d^{++}_D(v)=|N^{++}_D(v)|$, $d^{-}_D(v)=|N^{-}_D(v)|$ and $d^{--}_D(v)=|N^{--}_D(v)|$. We omit the subscript if the
digraph (resp. graph) is clear from the context. For short, we write $x\rightarrow y$ if the arc $(x,y)\in E$. 
We say that a vertex $v$ has the \emph{second neigborhood property} (SNP) if $d^{+}(v)\leq d^{++}(v)$.\\

\par In 1990, P. Seymour conjectured \cite{dean} the following statement: 

\begin{conjec}
 \textbf{(The Second Neighborhood Conjecture (SNC))}\\ Every digraph has a vertex with the SNP.
\end{conjec}

Seymour's conjecture restricted to tournaments is known as Dean's conjecture \cite{dean}. In 1996, Fisher \cite{fisher}
proved Dean's conjecture, thus asserting the SNC for tournaments. Another proof of Dean's conjecture was given by Thomassé
and Havet \cite{m.o.}, in 2000, using a tool called (local) median order. In 2007, Fidler and Yuster used also median orders to 
prove SNC for tournaments missing a matching. Ghazal proved the weighted version of SNC for tournaments missing a generalized star \cite{a}.\\

\par A median order $L=v_1v_2...v_n$ of a digraph $D$ is an order of the vertices of $D$ the maximizes
the size of the set of forward arcs of $D$, i.e., the set $\{(v_i,v_j)\in E; i<j\}$.
In fact, $L$ satisfies the feedback property: For all $1\leq i\leq j\leq n:$
$$ d^{+}_{[i,j]}(v_i)  \geq  d^{-}_{[i,j]}(v_i) $$
and 
$$ d^{-}_{[i,j]}(v_j) \geq  \omega d^{+}_{[i,j]}(v_j)  $$
where $[i,j]:=D[v_i,v_{i+1}, ...,v_j]$.\\
An order $L=v_1v_2...v_n$ satisfying the feedback property is called a local median order . The last vertex $v_n$
of a weighted local median order $L=v_1v_2...v_n$ of $D$ is called a \emph{feed} vertex of the digraph 
$D$ \cite{m.o.}.

\begin{thm}\cite{m.o.}
 Every feed vertex of a tournament has the SNP.
\end{thm}

\end{section}

\begin{section}{Dependency Digraph}

\hskip0.6cm Let $D=(V,E)$ be a digraph. For 2 vertices $x$ and $y$, we call $xy$ a missing edge if $(x,y) \notin E$ and
$(y,x)\notin E$. The missing graph $G$ of $D$ is the graph formed by the missing edges, formally, $E(G)$ is the set
of all the missing edge and $V(G)$ is the set of non whole vertices (vertices incident to some missing edges).
In this case, we say that $D$ is missing $G$.

\par we say that a missing edge $x_1y_1$ loses to a missing edge $x_2y_2$ if:
$x_1\rightarrow x_2$, $y_2\notin N^{+}(x_1)\cup N^{++}(x_1)$, $y_1\rightarrow y_2$ and $x_2\notin N^{+}(y_1)\cup N^{++}(y_1)$. 
The dependency digraph $\Delta$ of $D$ is defined as follows: Its vertex set consists of all the missing
edges and $(ab,cd)\in E(\Delta)$ if $ab$ loses to $cd$. Note that $\Delta$ may contain digons.\\

These digraphs were used in \cite{fidler} to prove SNC for tournaments missing a matching. However,
our defintion is general and is suitable for any digraph.\\

\begin{defn}\cite{a}
A missing edge $ab$ is called \emph{good} if:\\
$(i)$   $(\forall v \in V\backslash\{a,b\})[(v\rightarrow a)\Rightarrow(b\in N^{+}(v)\cup N^{++}(v))]$ or\\
$(ii)$ $(\forall v \in V\backslash\{a,b\})[(v\rightarrow b)\Rightarrow(a\in N^{+}(v)\cup N^{++}(v))]$.\\
If $ab$ satisfies $(i)$ we say that $(a,b)$ is a convenient orientation of $ab$.\\
If $ab$ satisfies $(ii)$ we say that $(b,a)$ is a convenient orientation of $ab$.
\end{defn}

The following holds by the definition of good missing edges and losing relation between them.

\begin{lem}
 Let $D$ be a digraph and let $\Delta$ denote its dependency digraph. A missing edge $ab$ is good
if and only if its in-degree in $\Delta$ is zero.
\end{lem}

Let $\mathcal{H}$ be a family of digraphs (digons are allowed) and let $G$ be a given graph. We say that $G$ is $\mathcal{H}$-forcing
if the dependency digraph of every digraph missing $G$ is a member of $\mathcal{H}$. The set of all $\mathcal{H}$-forcing graphs
is denoted by $\mathcal{F}(\mathcal{H})$.

A digraph is trivial if it has no arc.

\begin{prop}
 Let $\mathcal{H}$ be a family of digraphs. Then  $\mathcal{F}(\mathcal{H})$ is nonempty if and only if
$\mathcal{H}$ has a trivial digraph.
\end{prop}

\begin{proof}
 Let $G$ be a graph and let $D$ be any digraph missing it. Suppose $xy\rightarrow uv$ in $\Delta$, the dependency digraph of $D$,
namely $v\notin  N^{+}(x)\cup N^{++}(x)$. We add to $D$ an extra whole vertex $\alpha$ such that $x\rightarrow \alpha \rightarrow v$.
This breaks the arc $(xy,uv)$. Hence, by adding a sufficient number of such vertices, one obtains a digraph whose missing graph is
$G$ and such that its dependency digraph is trivial. This establishes the necessary condition.\\ The converse holds, 
by absorving that the dependency digraph of any digraph missing a star ( edges sharing only one endpoint) is trivial. 
\end{proof}

Let $\mathfrak{S}$ denote the class of all trivial digraphs. In \cite{a} Ghazal showed that the only $\mathfrak{S}$-forcing graphs
are generalized stars and proved that every digraph missing a generalized star satisfies Seymour's Second Neighborhood Conjecture.
In fact, a weighted version the following statement is proved in \cite{a}.

\begin{lem}
Let $D$ be a digraph. If all the missing edges of $D$ are good then it has a vertex with the SNP.
\end{lem}

\begin{prob}
 Let $\vec{\mathcal{P}}$ be the family of all digraphs composed of vertex disjoint directed paths only. Characterize
 $\mathcal{F}(\vec{\mathcal{P}})$.
\end{prob}

In the next section we present some classes of graphs contained in $\mathcal{F}(\vec{\mathcal{P}})$ and prove that every
digraph missing a member of these classes, satisfies SNC.

\end{section}

\begin{section}{Some Digraphs Missing Graphs of $\mathcal{F}(\vec{\mathcal{P})}$}

A comb G is a graph defined as follows:
\begin{description}
\item[1)] $V(G)$ is disjoint union of three set $A$, $X$ and $Y$.
\item[2)] $G[X\cup Y]$ is a complete graph.
\item[3)] $A$ is stable set with $N(A)=X$ and $N(a)\cap N(b) =\phi$ for any distinct vertices $a,b\in A$.
\item[4)] For every $a\in A$, $d(a)=1$.
\end{description}
Observe that the edges with an end in $A$ form a matching, say M.

\begin{prop}
 Combs are $\vec{\mathcal{P}}$-forcing.
\end{prop}

\begin{proof}

Let $D$ be a digraph missing a comb $G$. We follow the previous notations. The only possible arcs of $\Delta$
occurs between the edges in $M$. For $i=1,2,3$ let $a_ix_i\in M$ with $a_i\in A$ and $x_i\in X$. Suppose
$a_1x_1$ loses to the 2 others. Then we have $a_1\rightarrow x_3$, $x_1\rightarrow a_2$, $a_2\notin N^{++}(a_1)\cup N^{+}(a_1)$
and $x_3\notin N^{++}(x_1)\cup N^{+}(x_1)$. Since $a_2x_3$ is not a missing edge then either $a_2\rightarrow x_3$ or
$a_2\leftarrow x_3$. Whence, either $x_3\in N^{++}(x_1)\cup N^{+}(x_1)$ or $a_2\in N^{++}(a_1)\cup N^{+}(a_1)$. A contradiction.
Therefore, the maximimum out-degree in $\Delta$ is 1. Similarly, the maximum in-degree is 1. Thus $\Delta$ is composed of at most
vertex disjoint directed paths and directed cycles. Now it is enough to prove that it has no directed cycles.
Suppose that $C=a_0x0,a_1x_1,...,a_nx_n$ is a cycle. Then we have $a_{i+1}\notin N^{++}(a_i)$ and $a_i\leftarrow a_{i+1}$ for all
$i< n$. We prove, by induction on $i$, that $a_i\rightarrow a_n$ for all $i<n$. In particular, $a_{n-1}\rightarrow a_n$, a contradiction.
The case $i=1$ holds since $a_nx_n$ loses to $a_1x_1$. Now let $1<i<n$. By induction hypothesis, $a_{i-1}\rightarrow a_n$. 
Since $a_i\notin N^{++}(a_{i-1})$ and $a_ia_n$ is not a missing edge we must have $(a_i,a_n)\in D$.  

\end{proof}

\begin{thm}
 Every digraph missing a comb satisfies SNC.
\end{thm}

\begin{proof}

 Let $D$ be a digraph missing a comb $G$. We follow the previous notations. Let $P=a_0x_0,a_1x_1,...$ be a maximal
directed path in $\Delta$ ($a_i\in A$ ). By lemma 1, $a_0x_0$ has a convenient orientation. Suppose $(a_0,x_0)$ is a convenient 
orientation. In this case add $(a_{2i},x_{2i})$ and $(x_{2i+1},a_{2i+1})$ to $D$. 
Otherwise, we orient in the reverse direction. We do this for all such paths of $\Delta$. The obtained digraph $D'$
is missing the complete graph $G[X\cup Y]$. Clearly, all the missing edges of $D'$ are good (in $D'$), so we give
each one a convenient orientation and add it to $D'$. The obtained digraph $T$ is a tournament. Let $L$ be a local
median order of $T$ and let $f$ denote its feed vertex. By theorem 1, $f$ has the SNP in $T$. We claim that $f$ has the SNP
in $D$ as well.\\

Suppose $f$ is a whole vertex. We show that $f$ gains no vertex in its second out-neighborhood and hence our claim holds.
Assume $f\rightarrow u\rightarrow v\rightarrow f$ in $T$. Since $f$ is whole, $f\rightarrow u$ in $D$. 
If $u\rightarrow v$ in $D'-D$, then it is either a convenient orientation and hence $v\in N^{++}(f)$ or
there is a missing edge $rs$ that loses to $uv$, namely $s\rightarrow v$ and $u\notin N^{+}(s)\cup N^{++}(s)$.
However, $fs$ is not a missing edge, then we must have $f\rightarrow s$. Whence $v\in N^{++}(f)$. Now, if 
$u\rightarrow v$ in $T-D'$ then $v\in N^{++}_{D'}(f)$. But this case is already discussed. This argument is
used implicitly in the rest of the proof. \\

Suppose $f\in A$. There is a maximal directed path $P=a_0x_0,...,a_ix_i,...,a_kx_k$ with $f=a_i$.
If $(x_i,a_i)\in D'$ then $d^+(f)=d^+_T(f)\leq d^{++}_T(f)=d^{++}(f)$. In fact $f$ gains no new first nor second out-neighbor.
Otherwise $(a_i,x_i)\in D'$. If $i<k$, $f$ gains only $x_i$ (resp. $a_{i+1}$ ) as a first (resp. second ) out-neighbor.
If $i=k$, we reorient $a_kx_k$ as $(x_k,a_k)$. The same order $L$ is also a local median order of $T'$ the modified tournament.
Now $f$ gains no vertex in its second out-neighborhood.\\

Suppose $f\in X$. There is a maximal directed path $P=a_0x_0,...,a_ix_i,...,a_kx_k$ with $f=x_i$.
If $(a_i,x_i)\in D'$ we reorient all the missing edges incident to $x_i$ towards $x_i$. In this case
$f$ gains no new first nor second out-neighbor in the modified tournament.
Otherwise $(x_i,a_i)\in D'$. 
 If $i=k$, we reorient all the missing edges incident to $x_i$ towards $x_i$. In this case
$f$ gains no new first nor second out-neighbor in the modified tournament.
 If $i<k$, we reorient all the missing edges incident to $x_i$ towards $x_i$ except $(x_i,a_i)$.
In this case $f$ gains only $a_{i}$ (resp. $x_{i+1}$ ) as a first (resp. second ) out-neighbor in the modified tournament.\\

Suppose $f\in Y$. Reoreient all the missing edges incident to $y$ towards $y$. In the modified tournament
$f$ gains no vertex in its second out-neighborhood.\\

Therefore $D$ satisfies SNC.

\end{proof}

A $\tilde{K}^4$ is a graph obtained from the complete graph by removing 2 non adjacent edges.
If $xy$ and $uv$ are the removed edges then $\tilde{K}^4$ restricted to $\{x,y,u,v\}$ is a cycle of length 4.

\begin{prop}
The graphs $\tilde{K}^4$ are $\vec{\mathcal{P}}$-forcing.
\end{prop}

\begin{proof}
 This is clear because the dependency digraph can have at most one arc. 
\end{proof}

\begin{thm}
 Every digraph whose missing graph is a $\tilde{K}^4$ satisfies SNC.
\end{thm}

\begin{proof}
 Let $D$ be a digraph missing a $\tilde{K}^4$. If $\Delta$ has no arc then $D$
satisfies SNC by lemma 2. Otherwise, it has exactly one arc, say $xy\rightarrow uv$
with $x\rightarrow u$ and $v\notin N^{++}(v)$. Note that the cycle $C=xyuv$ is an induced cycle
in the missing graph. We may suppose that $(x,y)$ is a convenient orientation. Add $(x,y)$
and $(u,v)$ to $D$. The rest of the missing edges are good missing edges. So we give them a
convenient orientation and add to $D$. The obtained digraph $T$ is a tournament. Let $L$
be a local median order of $T$ and let $f$ denote its feed vertex. Now $f$ has the SNP in $T$. We discuss according to $f$.\\

Suppose $f$ is a whole vertex. Then $f$ gains no vertex in its second out-neighborhood.\\

Suppose $f=x$. Reorient all the missing edges incident to $x$ towards $x$ except $(x,y)$.
The same order $L$ is a local median order of the modified tournament $T'$. The only new first (resp. second) out-neighbor
of $f$ is $y$ (resp. $v$).\\

Suppose $f=y$, $u$, $v$ or a non whole vertex that does not belong $C$. Reorient all the missing edges incident to $f$ towards $f$. 
In the modified tournament, $f$ gains no vertex in its second out-neighborhood.\\

\end{proof}

A $\tilde{K}^5$ is a graph obtained from the complete graph by removing a cycle of length 5.
Note that $\tilde{K}^5$ restricted to the vertices of the removed cycle is also a cycle of
length 5.\\

In the following $ab\rightarrow cd$ means $ab$ loses to $cd$, namely, $a\rightarrow c$ and $b\rightarrow d$ (the order of the endpoints is considered).
Let $D$ be a digraph missing $\tilde{K}^5$ and let $\Delta$ denote its dependency digraph. Let $C=xyzuv$ be the induced cyle
of length 5 in $\tilde{K}^5$. Checking by cases, we find that $\Delta$ has at most 3 arcs.
If $\Delta$ has exactly 3 arcs then its arcs are (isomorphic to) $uv\rightarrow xy\rightarrow zu\rightarrow vx$ or
$uv\rightarrow xy\rightarrow zu$ and $xv\rightarrow zy$.\\
If $\Delta$ has exactly 2 arcs then they are (isomorphic to) $uv\rightarrow xy\rightarrow zu$ or $uv\rightarrow xy$ and $vx\rightarrow yz$.\\
If $\Delta$ has exactly 1 arc then it is (isomorphic to) $uv\rightarrow xy$. So we have the following.

\begin{prop}
The graphs $\tilde{K}^5$ are $\vec{\mathcal{P}}$-forcing.
\end{prop}

\begin{thm}
 Every digraph whose missing graph is a $\tilde{K}^5$ satisfies SNC.
\end{thm}

\begin{proof}
 Let $D$ be a digraph missing a $\tilde{K}^5$. Let $C=xyzuv$ be the induced cyle
of length 5 in $\tilde{K}^5$. If $\Delta$ has no arcs then $D$ satisfies SNC by lemma 2.\\

Suppose $\Delta$ has exactly one arc $uv\rightarrow xy$. Without loss of generality, we may assume that $(u,v)$
is a convenient orientation. Add $(u,v)$ and $(x,y)$ to $D$. We give the rest of the missing edges (they are good)
convenient orientatins and them add to $D$. Let $L$ be a local median order of the obtained tournament $T$ and let $f$ denote
its feed vertex. $f$ has the SNP in $T$. If $f=u$ the only new first (resp. second) out-neighbor of $u$ is $v$ (resp. $y$). 
Whence $f$ has the SNP in $D$. Otherwise, we reorient all the missing edges incident to $f$ towards $f$, if any exist.
The same order $L$ is a local median order of the new tournament $T'$ and $f$ has the SNP in $T'$. However, 
$f$ gains neither a new first out-neighbor nor a new second out-neighbor. So $f$ has the SNP in $D$.\\

Suppose $\Delta$ has exactly 2 arcs, say $uv\rightarrow xy$ and $vx\rightarrow yz$. We may assume
that $(u,v)$ is a convenient orientation. Add $(u,v)$ and $(x,y)$ to $D$. If $(v,x)$ is a convenient
orientation, we add $(v,x)$ and $(y,z)$ to $D$, otherwise we add their reverse. We give the rest of the missing edges (they are good)
convenient orientatins and add them to $D$. Let $L$ be a local median order of the obtained tournament $H$ and let $f$ denote
its feed vertex. We reorient every missing edge incident to $f$, whose other endpoint is not in $\{u,v,x\}$, towards $f$ if any exists . 
The same $L$ is a local median order of the new tournament $T$ and $f$ has the SNP in $T$.

If $f\notin \{u,v,x\}$ then it gains neither a new first out-neighbor nor a new second out-neighbor. So $f$ has the SNP in $D$.
 
If $f=u$, then the only new first (resp. second ) out-neighbor of $f$ is $v$ (resp. $y$), whence $f$ has the SNP in $D$.
 
If $f=v$ either $v\rightarrow x$ in $T$ and in this case the only new first (resp. second ) out-neighbor of $v$ is $x$ (resp. $z$) 
or $x\rightarrow v$ and in this case $f$ gains neither a new first out-neighbor nor a new second out-neighbor. Whence $f$ has the SNP in $D$.

If $f=x$ we reorient $xy$ as $(y,x)$. The same $L$ is a local median order of the new tournament $T'$ and $f$ has the SNP in $T'$.
If $v\rightarrow x$ in $T'$ then $f$ gains neither a new first out-neighbor nor a new second out-neighbor. Otherwis, $x\rightarrow v$
in $T'$ then the only new first (resp. second ) out-neighbor of $f$ is $v$ (resp. $y$). Whence $f$ has the SNP in $D$.\\

Suppose $\Delta$ has exactly 2 arcs with $uv\rightarrow xy\rightarrow zu$.  We may assume
that $(u,v)$ is a convenient orientation. Add $(u,v)$, $(x,y)$ and $(z,u)$ to $D$. We give the rest of the missing edges (they are good)
convenient orientatins and add them to $D$. Let $L$ be a local median order of the obtained tournament $H$ and let $f$ denote
its feed vertex. We reorient every missing edge incident to $f$, whose other endpoint is not in $\{u,v,x,y,z\}$, towards $f$ if any exists . 
The same $L$ is a local median order of the new tournament $T$ and $f$ has the SNP in $T$. 

If $f\notin \{u,v,x,y,z\}$ then it gains neither a new first out-neighbor nor a new second out-neighbor. So $f$ has the SNP in $D$.

If $f=u$, then the only new first (resp. second ) out-neighbor of $f$ is $v$ (resp. $y$), whence $f$ has the SNP in $D$.

If $f=v$ we orient $xv$ as $(x,v)$. The same $L$ is a local median order of the new tournament $T'$ and $f$ has the SNP in $T'$.

If $f=x$ we orient $xv$ as $(v,x)$. The same $L$ is a local median order of the new tournament $T'$ and $f$ has the SNP in $T'$.
The only new first (resp. second) out-neighbor of $f$ is $y$ (resp. $u$). Whence $f$ has the SNP in $D$.

If $f=y$ we orient $yz$  as $(z,y)$. The same $L$ is a local median order of the new tournament $T'$ and $f$ has the SNP in $T'$.
$f$ gains neither a new first out-neighbor nor a new second out-neighbor. So $f$ has the SNP in $D$.

If $f=z$ we orient $yz$ and $zu$ towards $z$. The same $L$ is a local median order of the new tournament $T'$ and $f$ has the SNP in $T'$.
$f$ gains neither a new first out-neighbor nor a new second out-neighbor. So $f$ has the SNP in $D$.\\

Suppose $\Delta$ has exactly 3 arcs with $uv\rightarrow xy\rightarrow zu\rightarrow vx$. We may assume
that $(u,v)$ is a convenient orientation. Add $(u,v)$, $(x,y)$, $(z,u)$ and $(v,x)$ to $D$. We give the rest of the missing edges (they are good)
convenient orientatins and then add to $D$. Let $L$ be a local median order of the obtained tournament $H$ and let $f$ denote
its feeed vertex. We reorient every missing edge incident to $f$, whose other endpoint is not in $\{u,v,x,y,z\}$, towards $f$ if any exists . 
The same $L$ is a local median order of the new tournament $T$ and $f$ has the SNP in $T$. 

If $f\notin \{u,v,x,y,z\}$ then it gains neither a new first out-neighbor nor a new second out-neighbor. So $f$ has the SNP in $D$.

If $f=u$, then the only new first (resp. second ) out-neighbor of $f$ is $v$ (resp. $y$), whence $f$ has the SNP in $D$.

If $f=v$ we orient $xv$ as $(x,v)$. The same $L$ is a local median order of the new tournament $T'$ and $f$ has the SNP in $T'$.
In this case $f$ gains neither a new first out-neighbor nor a new second out-neighbor. So $f$ has the SNP in $D$.

If $f=x$, the only new first (resp. second ) out-neighbor of $f$ is $y$ (resp. $u$). Whence $f$ has the SNP in $D$.

If $f=y$ we orient $yz$ as $(z,y)$. The same $L$ is a local median order of the new tournament $T'$ and $f$ has the SNP in $T'$.
$f$ gains neither a new first out-neighbor nor a new second out-neighbor. So $f$ has the SNP in $D$.

If $f=z$ we orient $yz$ towards $z$. The same $L$ is a local median order of the new tournament $T'$ and $f$ has the SNP in $T'$.
The only new first (resp. second) out-neighbor of $f$ is $u$ (resp. $x$). Whence $f$ has the SNP in $D$.\\

Finally, suppose $\Delta$ has exactly 3 arcs with $uv\rightarrow xy\rightarrow zu$ and $xv\rightarrow zy$. We may assume
that $(u,v)$ is a convenient orientation. Add $(u,v)$, $(x,y)$ and $(z,u)$ to $D$. Note that $xv$ is a good missing edge.
If $(x,v)$ is a convenient orientation add it with $(z,y)$, otherwise we add the reverse of these arcs. We give the rest 
of the missing edges (they are good) convenient orientatins and add them to $D$. Let $L$ be a local median order of the 
obtained tournament $H$ and let $f$ denote its feeed vertex. We reorient every missing edge incident to $f$, whose other 
endpoint is not in $\{u,v,x,y,z\}$, towards $f$ if any exists. The same $L$ is a local median order of the new tournament 
$T$ and $f$ has the SNP in $T$. 

If $f\notin \{u,v,x,y,z\}$ then it gains neither a new first out-neighbor nor a new second out-neighbor. So $f$ has the SNP in $D$.

If $f=u$, then the only new first (resp. second ) out-neighbor of $f$ is $v$ (resp. $y$), whence $f$ has the SNP in $D$.

If $f=v$ either $x\rightarrow f=v$ in $T$ and in this case it gains neither a new first out-neighbor nor a new second out-neighbor
or $f=v\rightarrow x$ and in this case the only new first (resp. second ) out-neighbor of $f$ is $x$ (resp. $z$). Whence $f$ has the SNP in $D$.

If $f=x$ either $v\rightarrow f=x$ in $T$ and in this case the only new first (resp. second ) out-neighbor of $f$ is $y$ (resp. $z$)
or $f=x\rightarrow v$ and in this case the only new first (resp. second ) out-neighbor of $f$ are $y$ and $v$ (resp. $z$). 
Whence $f$ has the SNP in $D$.

If $f=y$ or $z$, we orient the missing edges incident to $f$ towards $f$. The same $L$ is a local median order of the new tournament 
$T'$ and $f$ has the SNP in $T'$. $f$ gains neither a new first out-neighbor nor a new second out-neighbor. So $f$ has the SNP in $D$.

\end{proof}

Digraphs missing a matching are the digraphs with minimum degree $|V(D)|-2$. These digraphs satisfies SNC \cite{fidler}.
A more general class of digraphs is the class of digraphs with minimum degree at least $|V(D)|-3$. The
missing graph of such a digraph is composed of vertex disjoint directed paths and directed cycles.
$P_3$ is the path of length 3 and $C_3$, $C_4$ and $C_5$ are the cycles of length 3, 4 and 5 respectively.
Theorems 2, 3 and 4 implies the following.

\begin{cor}
 Every digraph whose missing graph is $P_3$, $C_3$, $C_4$ or a $C_5$ satisfies SNC.  
\end{cor}

\end{section}

\end{document}